\documentclass[11pt,oneside,english]{amsart}
\usepackage[T1]{fontenc}
\usepackage[latin9]{inputenc}
\usepackage[a4paper]{geometry}
\usepackage{color}
\usepackage{slashed}
\usepackage{amstext}
\usepackage{amsthm}
\usepackage{amssymb}
\usepackage{setspace}
\makeatletter
\numberwithin{equation}{section}
\numberwithin{figure}{section}
\theoremstyle{plain}
\newtheorem{thm}{\protect\theoremname}[section]
  \theoremstyle{definition}
  \newtheorem{defn}[thm]{\protect\definitionname}
  \theoremstyle{remark}
  \newtheorem*{rem*}{\protect\remarkname}
  \theoremstyle{definition}
  \newtheorem{condition}[thm]{\protect\conditionname}
  \theoremstyle{remark}
  \newtheorem*{acknowledgement*}{\protect\acknowledgementname}
  \theoremstyle{plain}
  \newtheorem{lem}[thm]{\protect\lemmaname}
  \theoremstyle{remark}
  \newtheorem{rem}[thm]{\protect\remarkname}
  \theoremstyle{plain}
  \newtheorem{prop}[thm]{\protect\propositionname}
  \theoremstyle{definition}
  \newtheorem*{example*}{\protect\examplename}
  \theoremstyle{plain}
  \newtheorem{cor}[thm]{\protect\corollaryname}

\makeatother

\usepackage{babel}
  \providecommand{\acknowledgementname}{Acknowledgement}
  \providecommand{\conditionname}{Condition}
  \providecommand{\corollaryname}{Corollary}
  \providecommand{\definitionname}{Definition}
  \providecommand{\examplename}{Example}
  \providecommand{\lemmaname}{Lemma}
  \providecommand{\propositionname}{Proposition}
  \providecommand{\remarkname}{Remark}
\providecommand{\theoremname}{Theorem}

\begin{document}

\title{\textup{On rational singularities and counting points of schemes
over finite rings}}

\author{Itay Glazer}

\address{Faculty of Mathematics and Computer Science, Weizmann Institute of
Science, 234 Herzl Street, Rehovot 76100, Israel.}

\email{itay.glazer@weizmann.ac.il}
\begin{abstract}
\begin{singlespace}
We study the connection between the singularities of a finite type
$\mathbb{Z}$-scheme $X$ and the asymptotic point count of $X$ over
various finite rings. In particular, if the generic fiber $X_{\mathbb{Q}}=X\times_{\mathrm{Spec}\mathbb{Z}}\mathrm{Spec}\mathbb{Q}$
is a local complete intersection, we show that the boundedness of
$\frac{\left|X(\mathbb{Z}/p^{n}\mathbb{Z})\right|}{p^{n\mathrm{dim}X_{\mathbb{Q}}}}$
in $p$ and $n$ is in fact equivalent to the condition that $X_{\mathbb{Q}}$
is reduced and has rational singularities. This paper completes the
main result in \cite{AA} (see \cite[Theorem 3.0.3]{AA}). 
\end{singlespace}
\end{abstract}

\maketitle
\begin{small}\tableofcontents{}

\end{small}

\raggedbottom

\section{Introduction}

\subsection{Motivation }

Given a finite type $\mathbb{Z}$-scheme $X$, the study of the quantity
$\left|X(\mathbb{Z}/m\mathbb{Z})\right|$ and its asymptotic behavior
is a fundamental question in number theory. The case when $m=p$,
or more generally the quantity $\left|X(\mathbb{F}_{q})\right|$ with
$q=p^{n}$, has been studied by many authors, most famously by Weil,
Lang, Dwork, Grothendieck and Deligne \cite{LW54,Dwo60,Gro65,Del74,Del80}.
The Lang-Weil estimates (see \cite{LW54}) give a good asymptotic
description of $\left|X(\mathbb{F}_{q})\right|$:
\[
\left|X(\mathbb{F}_{q})\right|=q^{\mathrm{dim}X_{\mathbb{F}_{q}}}(C_{X}+O(q^{-1/2})),
\]
 where $C_{X}$ is the number of top dimension irreducible components
of $X_{\overline{\mathbb{F}}_{q}}$ that are defined over $\mathbb{F}_{q}$.
From these estimates and the fact that 
\begin{equation}
\left|X(F)\right|=\left|U(F)\right|+\left|(X\backslash U)(F)\right|,
\end{equation}
for any open subscheme $U\subseteq X$ and any finite field $F$,
it follows that the asymptotics of $\left|X(\mathbb{F}_{p^{n}})\right|$,
in $p$ or in $n$, does not depend on the singularity properties
of $X$. For finite rings, however, $(1.1)$ is no longer true (e.g
$\left|\mathbb{A}^{1}(A)\right|=\left|A\right|$ and $\left|\left(\mathbb{A}^{1}-\{0\}\right)(A)\right|=\left|A^{\times}\right|$)
and indeed, the number $\left|X(\mathbb{Z}/m\mathbb{Z})\right|$ and
its asymptotics have much to do with the singularities of $X$. The
case when $m=p^{n}$ is a prime power was studied, among others, by
Borevich, Shafarevich, Denef, Igusa, du Sautoy-Grunewald and Mustata
(see \cite{Den91,dSG00,Igu00} and a recent overview at \cite{Mus}).

For a finite ring $A$, set $h_{X}(A):=\frac{\left|X(A)\right|}{\left|A\right|^{\mathrm{dim}X_{\mathbb{Q}}}}$.
If $X_{\mathbb{Q}}$ is smooth, one can show that for almost every
prime $p$, we have $h_{X}(\mathbb{Z}/p^{n}\mathbb{Z})=h_{X}(\mathbb{Z}/p\mathbb{Z})$
for all $n$, which by the Lang-Weil estimates is uniformly bounded.
On the other hand, if $X_{\mathbb{Q}}$ is singular, then $h_{X}(\mathbb{Z}/p^{n}\mathbb{Z})$
need not be bounded in $n$ or in $p$. The goal of this paper is
to investigate this phenomena and to complete the main result presented
in \cite{AA}, which we describe next.

\subsection{Related work}

In \cite{AA}, Aizenbud and Avni proved the following:
\begin{thm}
\label{Raminir theorem}\cite[Theorem 3.0.3]{AA} Let $X$ be a finite
type $\mathbb{Z}$-scheme such that $X_{\mathbb{Q}}$ is equi-dimensional
and a local complete intersection. Then the following are equivalent: 

i) For any $n$, $\mathrm{lim}_{p\rightarrow\infty}h_{X}(\mathbb{Z}/p^{n}\mathbb{Z})=1$.

ii) There exists a finite set of prime numbers $S$ and a constant
$C$, such that $\left|h_{X}(\mathbb{Z}/p^{n}\mathbb{Z})-1\right|<Cp^{-1/2}$
for any prime $p\notin S$ and any $n\in\mathbb{N}$.

iii) $X_{\mathbb{\overline{Q}}}$ is reduced, irreducible and has
rational singularities. 
\end{thm}

The following definition was introduced in \cite{AA16}: 
\begin{defn}
\cite[1.2, Definition II]{AA16} Let $X$ and $Y$ be smooth varieties
over a field $k$ of characteristic $0$. We say that a morphism $\varphi:X\rightarrow Y$
is $(FRS)$ if it is flat and any geometric fiber is reduced and has
rational singularities. We say that $\varphi$ is $(FRS)$ at $x\in X(k)$
if there exists a Zariski open neighborhood $U$ of $x$ such that
$U\times_{Y}\{\varphi(x)\}$ is reduced and has rational singularities.
\end{defn}

Aizenbud and Avni introduced an analytic criterion for a morphism
$\varphi$ to be $(FRS)$, which played a key role in the proof of
Theorem \ref{Raminir theorem}:
\begin{thm}
\label{Auxilery Raminir theorem} \cite[Theorem 3.4]{AA16} Let $\varphi:X\rightarrow Y$
be a map between smooth algebraic varieties defined over a finitely
generated field $k$ of characteristic $0$, and let $x\in X(k)$.
Then the following conditions are equivalent: 

a) $\varphi$ is $(FRS)$ at $x$.

b) There exists a Zariski open neighborhood $x\in U\subseteq X$,
such that for any non-Archimedean local field $F\supseteq k$ and
any Schwartz measure $m$ on $U(F)$, the measure $(\varphi|_{U(F)})_{*}(m)$
has continuous density (see Definition \ref{def:Schwartz measure}
for the notion of Schwartz/continuous density of a measure).

c) For any finite extension $k'/k$, there exists a non-Archimedean
local field $F\supseteq k'$ and a non-negative Schwartz measure $m$
on $X(F)$ that does not vanish at $x$ such that $\varphi_{*}(m)$
has continuous density. 
\end{thm}

\subsection{Main results}

In this paper, we generalize Theorem \ref{Raminir theorem} as follows: 
\begin{thm}
\label{Main theorem} Let $X$ be a finite type $\mathbb{Z}$-scheme
such that $X_{\mathbb{Q}}$ is equi-dimensional and a local complete
intersection. Then i), ii) and iii) in Theorem \ref{Raminir theorem}
are also equivalent to:

iv) $X_{\overline{\mathbb{Q}}}$ is irreducible and there exists $C>0$
such that $h_{X}(\mathbb{Z}/p^{n}\mathbb{Z})<C$ for any prime $p$
and any $n\in\mathbb{N}$.

v) $X_{\overline{\mathbb{Q}}}$ is irreducible and there exists a
finite set of primes $S$, such that for any $p\notin S$ , the sequence
$n\mapsto h_{X}(\mathbb{Z}/p^{n}\mathbb{Z})$ is bounded.
\end{thm}

\begin{rem*}
In fact, one can drop the demand that $X_{\overline{\mathbb{Q}}}$
is irreducible in conditions $iii),iv)$ and $v)$, such that they
will stay equivalent. For a slightly stronger statement- see Theorem
\ref{Main extended theorem}.
\end{rem*}
There are two main difficulties in the proof of Theorem \ref{Main theorem}.
The first one is portrayed in the fact that condition $v)$ seems
a-priori too weak, as it requires the bound on $h_{X}(\mathbb{Z}/p^{n}\mathbb{Z})$
to be uniform only in $n$, while in condition $ii)$, the demand
is that the bound is uniform both in $p$ and in $n$.

In order to show that condition $v)$ implies the other conditions,
we first reduce to the case when $X_{\mathbb{Q}}$ is a complete intersection
in an affine space, and thus can be written as the fiber at $0$ of
a morphism $\varphi:\mathbb{A}_{\mathbb{Q}}^{M}\rightarrow\mathbb{A}_{\mathbb{Q}}^{N}$,
which is flat above $0$. We can then translate condition $iii)$,
i.e the condition that $X_{\overline{\mathbb{Q}}}$ is reduced and
has rational singularities, to the condition that $\varphi:\mathbb{A}_{\mathbb{Q}}^{M}\rightarrow\mathbb{A}_{\mathbb{Q}}^{N}$
is (FRS) above $0$, i.e at any point $x\in\left(\varphi^{-1}(0)\right)(\overline{\mathbb{Q}})$.
After some technical argument, one can show that condition $v)$ implies
the following:
\begin{condition}
\label{cond:For-any-finite}For any finite extensions $k/\mathbb{Q}$
and $k'/k$, and any $x\in\left(\varphi^{-1}(0)\right)(k)$, there
exists a prime $p$ with $k'\hookrightarrow\mathbb{Q}_{p}$, $x\in\left(\varphi^{-1}(0)\right)(\mathbb{Z}_{p})$,
such that the sequence $n\mapsto\frac{\varphi_{*}(\mu)(p^{n}\mathbb{Z}_{p}^{N})}{p^{-nN}}$
is bounded, where $\mu$ is the normalized Haar measure on $\mathbb{Z}_{p}^{M}$.
\end{condition}

Hence, we would like to generalize Theorem \ref{Auxilery Raminir theorem},
such that Condition \ref{cond:For-any-finite} will imply the $(FRS)$
property of $\varphi$ above $0$. 

The measure $\varphi_{*}(m)$ as in Condition \ref{cond:For-any-finite}
is said to be bounded with respect to the local basis $\{p^{n}\mathbb{Z}_{p}^{N}\}_{n}$
for the topology of $\mathbb{Q}_{p}^{N}$ at $0$ (see Definition
\ref{def:3.1}).We introduce the notion of bounded eccentricity of
a local basis to the topology of an $F$-analytic manifold (Section
\ref{subsec:Local-basis-of}), and prove the following generalization
of Theorem \ref{Auxilery Raminir theorem}:
\begin{thm}
\label{Main auxilary theorem} Let $\varphi:X\rightarrow Y$ be a
map between smooth algebraic varieties defined over a finitely generated
field $k$ of characteristic $0$, and let $x\in X(k)$. Then a),
b), c) in Theorem \ref{Auxilery Raminir theorem} are also equivalent
to:

c') For any finite extension $k'/k$, there exists a non-Archimedean
local field $F\supseteq k'$ and a non-negative Schwartz measure $m$
on $X(F)$ that does not vanish at $x$, such that $\varphi_{*}(m)$
is bounded with respect to\textbf{ }some\textbf{ }local basis $\mathcal{N}$
of bounded eccentricity at $\varphi(x)$.
\end{thm}

We then use Theorem \ref{Main auxilary theorem} and the fact that
the local basis $\{p^{n}\mathbb{Z}_{p}^{N}\}_{n}$ is of bounded eccentricity
to show that $v)$ implies condition $iii)$. 

The second difficulty is to show that if $h_{X}(\mathbb{Z}/p^{n}\mathbb{Z})$
is bounded for almost any prime $p$, then it is in fact bounded for
any $p$. We first prove this for the case that $X$ is a complete
intersection in an affine space, denoted $(CIA)$ (Proposition \ref{Proposition 4.5}).
We then deal with the case when $X_{\mathbb{Q}}$ is a $(CIA)$, by
constructing a finite type $\mathbb{Z}$-scheme $\widehat{X},$ which
is a $(CIA)$ and a morphism $\psi:X\longrightarrow\widehat{X}$,
such that $\psi_{\mathbb{Q}}:X_{\mathbb{Q}}\longrightarrow\widehat{X}_{\mathbb{Q}}$
is an isomorphism (Lemma \ref{Lemma 4.6}). We prove this case by
showing the existence of $c,N\in\mathbb{N}$ such that
\[
\left|X(\mathbb{Z}/p^{n}\mathbb{Z})\right|\leq p^{Nc}\cdot\left|\widehat{X}(\mathbb{Z}/p^{n}\mathbb{Z})\right|,
\]
(Lemma \ref{Lemma 4.7}). For the general case, we first cover $X_{\mathbb{Q}}$
by affine $\mathbb{Q}$-schemes $\{U_{i}\}$ such that $U_{i}$ is
a $(CIA)$, and then consider a collection of $\mathbb{Z}$-schemes
$\{\widetilde{U}_{i}\}$, such that $\widetilde{U}_{i}\simeq U_{i}$
over $\mathbb{Q}$. Finally, using the explicit construction of $\widetilde{U}_{i}$
we show that
\[
h_{X}(\mathbb{Z}/p^{n}\mathbb{Z})\leq\sum_{i}h_{\widetilde{U}_{i}}(\mathbb{Z}/p^{n}\mathbb{Z}),
\]
and since $\left(\widetilde{U}_{i}\right)_{\mathbb{Q}}\simeq U_{i}$
is a $(CIA)$, we are done by the last case. 
\begin{acknowledgement*}
I would like to thank my advisor \textbf{Avraham Aizenbud} for presenting
me with this problem, teaching and helping me in this work. I hold
many thanks to \textbf{Nir Avni }for helpful discussions and for hosting
me at the Northwestern university in July 2016, during which a large
part of this work was done. I also thank \textbf{Yotam Hendel} for
fruitful talks. This work was partially supported by the ISF grant
{[}687/13{]}, the BSF grant {[}2012247{]} and the Minerva foundation
grant. 
\end{acknowledgement*}

\section{\label{sec:Preliminaries}Preliminaries}

In this section, we recall some definitions and facts in algebraic
geometry and $F$-analytic manifolds, for a non-Archimedean local
field $F$. Most of the the statements presented here can be found
in \cite[Appendix B in the arxiv version]{AA16} and in \cite{AA}.

\subsection{Preliminaries in algebraic geometry}

Let $A$ be a commutative ring. A sequence $x_{1},...,x_{r}\in A$
is called a \textit{regular sequence} if $x_{i}$ is not a zero-divisor
in $A/(x_{1},...,x_{i-1})$ for each $i$, and we have a proper inclusion
$(x_{1},...,x_{r})\subsetneq A$. If $(A,\mathfrak{m})$ is a Noetherian
local ring then the $depth$ of $A$, denoted $\mathrm{depth}(A)$,
is defined to be the length of the longest regular sequence with elements
in $\mathfrak{m}$. It follows from Krull\textquoteright s principal
ideal theorem that $\mathrm{depth}(A)$ is smaller or equal to $\mathrm{dim}(A)$,
the Krull dimension of $A$. A Noetherian local ring $(A,\mathfrak{m})$
is \textit{Cohen-Macaulay} if $\mathrm{depth}(A)=\mathrm{dim}(A)$.
A locally Noetherian scheme $X$ is said to be \textit{Cohen-Macaulay}
if for any $x\in X$, the local ring $\mathcal{O}_{X,x}$ is Cohen-Macaulay. 

Let $X$ be an algebraic variety over a field $k$. We say that $X$
has a \textit{resolution of singularities}, if there exists a proper
morphism $p:\widetilde{X}\rightarrow X$ such that $\widetilde{X}$
is smooth and $p$ is a birational equivalence. A \textit{strong resolution
of singularities} of $X$ is a resolution of singularities $p:\widetilde{X}\rightarrow X$
which is an isomorphism over the smooth locus of $X$, denoted $X^{sm}$.
It is a theorem of Hironaka \cite{Hir64}, that any variety $X$ over
a field $k$ of characteristic zero admits a strong resolution of
singularities $p:\widetilde{X}\rightarrow X$. 

For the following definition, see \cite[I.3 pages 50-51]{KKMS73}
or \cite[Definition 6.1]{AA16}; we say that $X$ has \textit{rational
singularities} if for any (or equivalently, for some) resolution of
singularities $p:\widetilde{X}\longrightarrow X$, the natural morphism
$\mathcal{O}_{X}\rightarrow\mathrm{R}p_{*}(\mathcal{O}_{\widetilde{X}})$
is a quasi-isomorphism, where $\mathrm{R}p_{*}$ is the higher direct
image. A point $x\in X(k)$ is a \textit{rational singularity} if
there exists a Zariski open neighborhood $U\subseteq X$ of $x$ that
has rational singularities.

We denote by $\Omega_{X}^{r}$ the sheaf of differential $r$-forms
on $X$ and by $\Omega_{X}^{r}[X]$ (resp. $\Omega_{X}^{r}(X)$) the
regular (resp. rational)  $r$-forms. The following lemma gives a
local characterization of rational singularities:
\begin{lem}
(see e.g \cite[Proposition 6.2]{AA16}) \textup{An affine $k$-variety
$X$ has rational singularities if and only if $X$ is Cohen\textendash Macaulay,
normal, and for any, or equivalently, some strong resolution of singularities
$p:\widetilde{X}\rightarrow X$ and any top differential form $\omega\in\Omega_{X^{sm}}^{top}[X^{sm}]$,
there exists a top differential form $\widetilde{\omega}\in\Omega_{\widetilde{X}}^{top}[\widetilde{X}]$
such that $\omega=\widetilde{\omega}|_{X^{sm}}$.} 
\end{lem}

Let $X$ be a finite type scheme over a ring $R$. Then $X$ is called:
\begin{enumerate}
\item A \textit{complete intersection $(CI)$} if there exists an affine
scheme $Y$, a smooth morphism $Y\rightarrow\mathrm{Spec}R$, a closed
embedding $X\hookrightarrow Y$ over $\mathrm{Spec}R$, and a regular
sequence $f_{1},...,f_{r}\in\mathcal{O}_{Y}(Y)$, such that the ideal
of $X$ in $Y$ is generated by the $\{f_{i}\}$. 
\item A\textit{ local complete intersection $(LCI)$} if there is an open
cover $\{U_{i}\}$ of $X$ such that each $U_{i}$ is a $(CI)$. 
\item A\textit{ complete intersection in an affine space $(CIA)$} if $X$
is a complete intersection in $Y$, with $Y=\mathbb{A}_{R}^{n}$ an
affine space. 
\item A\textit{ local complete intersection in an affine space $(LCIA)$}
if there is an open affine cover $\{U_{i}\}$ of $X$ such that each
$U_{i}$ is a $(CIA)$.
\end{enumerate}
\begin{rem}
\label{rem:For-an-affine} For an affine $k$-variety, the notion
of $(CIA)$ is not equivalent to $(CI)$ (e.g consider $X$ to be
any affine smooth $k$-variety which is not a $(CIA)$). On the other
hand, the notion of $(LCI)$ is equivalent to $(LCIA)$ for finite
type $k$-schemes. We will therefore use the notation $(LCI)$ for
both notions. 
\end{rem}

The following Proposition is a consequence of the above remark and
the Miracle Flatness Theorem (e.g \cite[Theorem 26.2.11]{Vak}) 
\begin{prop}
\label{Prop 2.3} Let $X$ be $k$-variety. If $X$ is an $(LCI)$
then there exists an open affine cover $\{U_{i}\}$ of $X$ and morphisms
$\varphi_{i},\psi_{i}$, where $\varphi_{i}:\mathbb{A}_{k}^{m_{i}}\longrightarrow\mathbb{A}_{k}^{n_{i}}$
is flat above $0$, and $\psi_{i}:U_{i}\hookrightarrow\mathbb{A}_{k}^{m_{i}}$
is a closed embedding that induces a $k$-isomorphism $\psi_{i}:U_{i}\simeq\varphi_{i}^{-1}(0)$.
\end{prop}

Using Proposition \ref{Prop 2.3} and \cite[Theorem 11.3.10]{Gro66},
one can obtain:
\begin{prop}
\label{prop 2.4} Let $X$ be a finite type $\mathbb{Z}$-scheme.
If $X$ is a $(CIA)$ then there exists $\mathbb{Z}$-morphisms $\varphi,\psi$,
where $\varphi:\mathbb{A}_{\mathbb{Z}}^{m}\longrightarrow\mathbb{A}_{\mathbb{Z}}^{n}$
is flat above $0$, and $\psi:X\hookrightarrow\mathbb{A}_{\mathbb{Z}}^{m}$
is an inclusion that induces a $\mathbb{Z}$-isomorphism $\psi:X\simeq\varphi^{-1}(0)$. 
\end{prop}

A commutative Noetherian ring $A$ is called \textit{Gorenstein} if
it has finite injective dimension as an $A$-module. A locally Noetherian
scheme $X$ is said to be \textit{Gorenstein} if all its local rings
are Gorenstein. Any locally Noetherian scheme $X$ which is a local
complete intersection is also Gorenstein. 

\subsection{Some facts on $F$-analytic manifolds}

Let $X$ be a $d$-dimensional smooth algebraic $k$-variety and $F\supseteq k$
be a non-Archimedean local field, with ring of integers $\mathcal{O}_{F}$.
Then $X(F)$ has a structure of an $F$-analytic manifold. Given $\omega\in\Omega_{X}^{top}(X)$,
we can define a measure $\left|\omega\right|_{F}$ on $X(F)$ as follows.
For a compact open set $U\subseteq X(F)$ and an $F$-analytic diffeomorphism
$\phi$ between an open subset $W\subseteq F^{d}$ and $U$, we can
write $\phi^{*}\omega=g\cdot dx_{1}\wedge...\wedge dx_{n}$, for some
$g:W\rightarrow F$, and define 
\[
\left|\omega\right|_{F}(U)=\int_{W}\left|g\right|_{F}d\lambda,
\]
where $\left|\,\right|_{F}$ is the normalized absolute value on $F$
and $\lambda$ is the normalized Haar measure on $F^{d}$. Note that
this definition is independent of the diffeomorphism $\phi$, and
that this uniquely defines a measure on $X(F)$. 
\begin{defn}
~\label{def:Schwartz measure}
\begin{enumerate}
\item A measure $m$ on $X(F)$ is called $smooth$ if every point $x\in X(F)$
has an analytic neighborhood $U$ and an $F$-analytic diffeomorphism
$f:U\rightarrow\mathcal{O}_{F}^{d}$ such that $f_{*}m$ is some Haar
measure on $\mathcal{O}_{F}^{d}$ . 
\item A measure on $X(F)$ is called $Schwartz$ if it is smooth and compactly
supported. 
\item We say that a measure $\mu$ on $X(F)$ has $continuous\,density$,
if there is a smooth measure $m$ and a continuous function $f:X(F)\rightarrow\mathbb{C}$
such that $\mu=f\cdot m$. 
\end{enumerate}
\end{defn}

The following proposition characterizes Schwartz measures and measures
with continuous density:
\begin{prop}
\cite[Proposition 3.3]{AA16} \label{Prop 2.6} Let $X$ be a smooth
variety over a non-Archimedean local field $F$. 

1) A measure $m$ on $X(F)$ is Schwartz if and only if it is a linear
combination of measures of the form $f\left|\omega\right|_{F}$, where
$f$ is a Schwartz function (i.e locally constant and compactly supported)
on $X(F)$, and $\omega\in\Omega_{X}^{top}(X)$ has no zeros or poles
in the support of $f$.

2) A measure $\mu$ on $X(F)$ has continuous density if and only
if for every point $x\in X(F)$ there is an analytic neighborhood
$U$ of $x$, a continuous function $f:U\rightarrow\mathbb{C}$, and
$\omega\in\Omega_{X}^{top}(X)$ with no poles in $U$ such that $\mu=f\left|\omega\right|_{F}$. 
\end{prop}

\begin{prop}
\cite[Proposition 3.5]{AA16} \label{Prop 2.7} Let $\varphi:X\rightarrow Y$
be a smooth map between smooth varieties defined over a non-Archimedean
local field $F$. 

(1) If $m$ is a Schwartz measure on $X(F)$, then $\varphi_{*}m$
is a Schwartz measure on $Y(F)$.

(2) Assume that $\omega_{X}\in\Omega_{X}^{top}[X]$ and $\omega_{Y}\in\Omega_{Y}^{top}[Y]$,
where $\omega_{Y}$ is nowhere vanishing, and that $f$ is a Schwartz
function on $X(F)$. Then the measure $\varphi_{*}(f\left|\omega_{X}\right|_{F})$
is absolutely continuous with respect to $\left|\omega_{Y}\right|_{F}$,
and its density at a point $y\in Y(F)$ is $\int_{\varphi^{-1}(y)(F)}f\cdot\left|\frac{\omega_{X}}{\varphi^{*}\omega_{Y}}|_{\varphi^{-1}(y)}\right|_{F}$.
\end{prop}

\section{An analytic criterion for the $(FRS)$ property }

Our goal in this section is to relax condition $c)$ of Theorem \ref{Auxilery Raminir theorem}.
We are motivated in proving the implication $v)\Longrightarrow iii)$
of Theorem \ref{Main theorem} and as we have seen at the introduction,
we wish to find a condition $c)'$ that is similar to Condition \ref{cond:For-any-finite},
such that it will imply the $(FRS)$ property (condition $(a)$ of
Theorem \ref{Auxilery Raminir theorem}).
\begin{defn}
\label{def:3.1} Let $F$ be a non-Archimedean local field, $X$ be
an $F$-analytic manifold and $\mu$ be a measure on $X$. Let $\mathcal{N}=\{N_{i}\}_{i\in I}$
be a local basis for the topology of $X$ at a point $x\in X$. We
say that $\mu$ is \textit{bounded with respect to $\mathcal{N}$},
if there exists a smooth measure $\lambda$ on $X$ and an open analytic
neighborhood $U$ of $x$, such that $\left|\frac{\mu(N_{i})}{\lambda(N_{i})}\right|$
is uniformly bounded on $\mathcal{N}_{U}:=\{N_{i}\in\mathcal{N}|N_{i}\subseteq U\}$.
\end{defn}

Let $\varphi:X\rightarrow Y$, $m$ and $F$ be as in Theorem \ref{Auxilery Raminir theorem}.
A possible relaxation $c')$ of $c)$, is to require $\varphi_{*}(m)$
to be bounded with respect to any local basis of the topology of $Y(F)$
at $\varphi(x)$. While this condition is equivalent to $a)$ and
$b)$ it is still to not weak enough for our purpose of proving Theorem
\ref{Main theorem}. A much weaker condition $c'')$ is to demand
that $\varphi_{*}(m)$ is bounded with respect to \textbf{some} local
basis at $\varphi(x)$. Unfortunately, the following example shows
that the latter demand is too weak: 
\begin{example*}
Consider the map $\varphi:\mathbb{A}_{\mathbb{Q}}^{2}\longrightarrow\mathbb{A}_{\mathbb{Q}}$
defined by $(x,y)\longmapsto x^{2}$. The fiber over $0$ is not reduced,
and thus $\varphi$ is not $(FRS)$ over $0$. Fix a finite extension
$k/\mathbb{Q}$ and embed $k$ in $\mathbb{Q}_{p}$ for some prime
$p$ (see Lemma \ref{Lemma 4.3}). Let $\lambda_{1},\lambda_{2}$
be the normalized Haar measure on $\mathbb{Q}_{p},\mathbb{Q}_{p}^{2}$
and let $m=1_{\mathbb{Z}_{p}^{2}}\cdot\lambda_{2}$ be a Schwartz
measure. Now consider the following collection $\mathcal{N}$ of sets
$B_{n}$ constructed as follows. Define $B_{n}^{1}:=\{x\in\mathbb{Z}_{p}|\left|x\right|\leq p^{-2n^{2}}\}$
and $B_{n}^{2}:=\{x\in\mathbb{Z}_{p}|\left|x-a_{n}\right|\leq p^{-4n}\}$,
where $a_{n}=p^{2n+1}$. Note that any $x\in B_{n}^{2}$ has norm
$p^{-2n-1}$ and thus is not a square, so $\varphi^{-1}(B_{n}^{2})=\slashed{O}$.
Denote $B_{n}=B_{n}^{1}\cup B_{n}^{2}$ and notice that $\mathcal{N}:=\{B_{n}\}_{n=1}^{\infty}$
is a local basis at $0$ and that: 
\[
\underset{n\rightarrow\infty}{\mathrm{lim}}\frac{\varphi_{*}m(B_{n})}{\lambda_{1}(B_{n})}=\underset{n\rightarrow\infty}{\mathrm{lim}}\frac{m(\varphi^{-1}(B_{n}^{1}))}{p^{-2n^{2}}+p^{-4n}}=\underset{n\rightarrow\infty}{\mathrm{lim}}\frac{p^{-n^{2}}}{p^{-2n^{2}}+p^{-4n}}\longrightarrow0.
\]

This shows that $\varphi$ satisfies condition $c)''$ but is not
$(FRS)$ at $(0,0)$.
\end{example*}
Luckily, we can relax $c)$ by demanding that $\varphi{}_{*}(m)$
is bounded with respect to \textbf{some} local basis at $\varphi(x)$,
if this basis is nice enough. In order to define precisely what we
mean, we introduce the notion of a local basis of bounded eccentricity. 

\subsection{\label{subsec:Local-basis-of} Local basis of bounded eccentricity}
\begin{defn}
Let $F$ be a local field, and $\lambda$ be a Haar measure on $F^{n}$.
\begin{enumerate}
\item A collection of sets $\mathcal{N}=\{N_{i}\}_{i\in I}$ in $F^{n}$
is said to have a \textit{bounded eccentricity} at $x\in F^{n}$,
if there exists a constant $C>0$ such that $\underset{i}{\mathrm{Sup}}\frac{\lambda(B_{min_{i}}(x))}{\lambda(B_{max_{i}}(x))}\leq C$,
where $B_{max_{i}}(x)$ is the maximal ball around $x$ that is contained
in $N_{i}$ and $B_{min_{i}}(x)$ is the minimal ball around $x$
that contains $N_{i}$.
\item We call $\mathcal{N}=\{N_{i}\}_{i\in\alpha}$ a \textit{local basis
of bounded eccentricity} at $x$, if it is a local basis of the topology
of $F^{n}$ at $x$, and there exists $\epsilon>0$, such that $\mathcal{N}_{\epsilon}:=\{N_{i}\in\mathcal{N}|N_{i}\subseteq B_{\epsilon}(x)\}$
has bounded eccentricity. 
\end{enumerate}
\end{defn}

\begin{rem*}
Note that $\mathcal{N}_{\epsilon}\neq\slashed{O}$ for any $\epsilon>0$
since it is a local basis at $x$. 
\end{rem*}
\begin{lem}
\label{Lemma 3.3}Let $\phi:F^{n}\longrightarrow F^{n}$ be an $F$-analytic
diffeomorphism. Let $\mathcal{N}=\{N_{i}\}_{i\in\alpha}$ be a local
basis of bounded eccentricity at $x\in F^{n}$. Then $\phi(\mathcal{N})$
is a local basis of bounded eccentricity at $\phi(x)$.
\end{lem}

\begin{proof}
Let $d\phi_{x}=A$ be the differential of $\phi$ at $x$. Since $\phi$
is a diffeomorphism, then for any $C>1$, there exists $\delta,\delta'>0$
such that for any $y\in B_{\delta}(x)$:
\[
\frac{1}{C}<\frac{\left|\phi(y)-\phi(x)\right|_{F}}{\left|A\cdot(y-x)\right|_{F}}<C,
\]
and for any $z\in B_{\delta'}(\phi(x))$ we have:
\[
\frac{1}{C}<\frac{\left|\phi^{-1}(z)-x\right|_{F}}{\left|A^{-1}\cdot(z-\phi(x))\right|_{F}}<C.
\]

We can choose small enough $\delta,\delta'$ such that $\mathcal{N}_{\delta}$
is a collection of sets of bounded eccentricity and $\phi(\mathcal{N}_{\delta})\supseteq\phi(\mathcal{N})_{\delta'}$.
We now claim that $\mathcal{M}_{\delta'}:=\phi(\mathcal{N})_{\delta'}$
is a collection of sets of bounded eccentricity at $\phi(x)$. Let
$B_{min_{i}}(x)$ be the minimal ball that contains $N_{i}\in\mathcal{N}_{\delta}$
and $B_{max_{i}}(x)$ be the maximal ball that is contained in $N_{i}$.
Notice that for any $y\in B_{min_{i}}(x)\subseteq B_{\delta}(x)$
we have 
\[
\left|\phi(y)-\phi(x)\right|_{F}<C\cdot\left|A\cdot(y-x)\right|_{F}\leq C\cdot\left\Vert A\right\Vert \cdot\left|y-x\right|_{F}\leq C\cdot min_{i}\cdot\left\Vert A\right\Vert ,
\]
thus $\phi(N_{i})\subseteq\phi(B_{min_{i}}(x))\subseteq B_{C\cdot min_{i}\cdot\left\Vert A\right\Vert }(\phi(x))$.
Similarly, for any $z\in B_{max_{i}/(C\cdot\left\Vert A^{-1}\right\Vert )}(\phi(x))$
we have that 
\[
\left|\phi^{-1}(z)-x\right|_{F}<C\cdot\left|A^{-1}\cdot(z-\phi(x))\right|_{F}\leq C\cdot\left\Vert A^{-1}\right\Vert \cdot\frac{max_{i}}{C\cdot\left\Vert A^{-1}\right\Vert }=max_{i}.
\]

Therefore, $\phi^{-1}(B_{max_{i}/(C\cdot\left\Vert A^{-1}\right\Vert )}(\phi(x)))\subseteq B_{max_{i}}(x)\subseteq N_{i}$
and thus $B_{max_{i}/(C\cdot\left\Vert A^{-1}\right\Vert )}(\phi(x))\subseteq\phi(N_{i})$.
Thus we get that 
\[
B_{max_{i}/(C\cdot\left\Vert A^{-1}\right\Vert )}(\phi(x))\subseteq\phi(N_{i})\subseteq B_{C\cdot min_{i}\cdot\left\Vert A\right\Vert }(\phi(x)),
\]
 for any $i$. By assumption, there exists some $D>0$ such that $\frac{\lambda(B_{min_{i}}(x))}{\lambda(B_{max_{i}}(x))}<D$
for any set $N_{i}\in\mathcal{N}_{\delta}$. Hence: 
\[
\frac{\lambda(B_{C\cdot min_{i}\cdot\left\Vert A\right\Vert }(\phi(x)))}{\lambda(B_{max_{i}/(C\cdot\left\Vert A^{-1}\right\Vert )}(\phi(x)))}\leq C^{2n}\left\Vert A\right\Vert ^{n}\cdot\left\Vert A^{-1}\right\Vert ^{n}\cdot D,
\]
and $\mathcal{M}_{\delta'}$ has bounded eccentricity.
\end{proof}
Lemma \ref{Lemma 3.3} implies that the following definition is well
defined:
\begin{defn}
Let $X$ be an $F$-analytic manifold and $\lambda$ be a Haar measure
on $F^{n}$.
\begin{enumerate}
\item A local basis $\mathcal{N}$ at $x\in X$ is said to have a \textit{bounded
eccentricity} if given an $F$-analytic diffeomorphism $\phi$ between
an open subset $W\subseteq F^{n}$ and an open neighborhood $U$ of
$x$, we have that $\mathcal{\widetilde{N}}=\{\phi^{-1}(N)|N\in\mathcal{N},\,N\subseteq U\}$
is a local basis of bounded eccentricity. 
\item A measure $m$ on $X$ is said to be $\mathcal{N}$\textit{-bounded},
if there exists $\epsilon>0$ such that:
\[
\underset{N\in\mathcal{N}_{\epsilon}}{\mathrm{sup}}\frac{m(N)}{\lambda(N)}<\infty.
\]
\end{enumerate}
\end{defn}

\subsection{Proof of Theorem \ref{Main auxilary theorem}}

It is easy to see that $c)\Longrightarrow c')$. The proof of the
implication $c')\Longrightarrow a)$ is a variation of the proof of
$c)\Longrightarrow a)$ of Theorem \ref{Auxilery Raminir theorem}
(see \cite[Section 3.7]{AA16}). Let $k$ be a finitely generated
field of characteristic $0$, $\varphi:X\longrightarrow Y$ be a morphism
of smooth $k$-varieties $X,Y$ and let $x\in X(k)$. Assume that
condition $c')$ of Theorem \ref{Main auxilary theorem} holds. Let
$Z=\varphi^{-1}(\varphi(x))$ and denote by $X^{S}$ the smooth locus
of $\varphi$. The following lemma is a slight variation of \cite[Claim 3.19]{AA16}.
Since we use the constructions presented in the proof of \cite[Claim 3.19]{AA16},
and for the convenience of the reader, we write the full steps and
use similar notations as well.
\begin{lem}
There exists a Zariski neighborhood $U$ of $x$ such that $Z\cap X^{S}\cap U$
is a dense subvariety of $Z\cap U$. 
\end{lem}

\begin{proof}
Let $Z_{1},...,Z_{n}$ be the absolutely irreducible components of
$Z$ containing $x$. After restricting to an open neighborhood of
$x$ that does not intersect the other irreducible components, it
is enough to show that $Z_{i}\cap X^{S}$ is Zariski dense in $Z_{i}$
for any $i$. Since $X^{S}$ is open, it is enough to show that $Z_{i}\cap X^{S}$
is non-empty for any $i$. 

Assume that $Z_{i}\cap X^{S}=\slashed{O}$ for some $i$. Then $\mathrm{dimker}d\varphi_{z}>\mathrm{dim}X-\mathrm{dim}Y$
for any $z\in Z_{i}(\overline{k})$. By the upper semi-continuity
of $\mathrm{dimker}d\varphi$, there is a non-empty open set $W_{i}\subseteq Z_{i}$
and an integer $r\ge1$ such that $\mathrm{dim}\mathrm{ker}\mathrm{d}\varphi|_{z}=\mathrm{dim}X-\mathrm{dim}Y+r$
for all $z\in W_{i}(\overline{k})$ and such that $W_{i}\cap Z_{j}=\slashed{O}$
for any $j\neq i$. Let $k'/k$ be a finite extension such that both
$Z_{i},W_{i}$ are defined over $k'$ and $W_{i}^{sm}(k')\neq\slashed{O}$.
By \cite[Lemma 3.14]{AA16}, we can choose $k'$ such that $x\in\overline{W_{i}^{sm}(F)}$
for any non-Archimedean local field $F\supseteq k'$. 

By our assumption, there exists a non-Archimedean local field $F\supseteq k'$
and a non-negative Schwartz measure $m$ on $X(F)$ that does not
vanish at $x$ and such that $\varphi_{*}m$ is bounded with respect
to some local basis $\mathcal{N}$ (at $\varphi(x)$) of bounded eccentricity.
Since $x\in\overline{W_{i}^{sm}(F)}$, there exists a point $p\in W_{i}^{sm}(F)\cap\mathrm{supp}(m)$. 

By the implicit function theorem, there exist neighborhoods $U_{X}\subseteq X(F)$
and $U_{Y}\subseteq Y(F)$ of $p$ and $\varphi(x)=\varphi(p)$ respectively,
analytic diffeomorphisms $\alpha_{X}:U_{X}\rightarrow\mathcal{O}_{F}^{\mathrm{dim}X},\,\alpha_{Y}:U_{Y}\rightarrow\mathcal{O}_{F}^{\mathrm{dim}Y}$
and $\alpha_{Z_{i}}:U_{X}\cap W_{i}^{sm}(F)\rightarrow\mathcal{O}_{F}^{\mathrm{dim}Z_{i}}$
such that $\alpha_{X}(p)=0$, $\alpha_{Y}(\varphi(p))=0$, and an
analytic map $\psi:\mathcal{O}_{F}^{\mathrm{dim}X}\longrightarrow\mathcal{O}_{F}^{\mathrm{dim}Y}$
such that the following diagram commutes: 
\[
\begin{array}{ccccc}
U_{X}\cap W_{i}^{sm}(F) & \hookrightarrow & U_{X} & \overset{\varphi|_{U_{X}}}{\longrightarrow} & U_{Y}\\
\downarrow\alpha_{Z_{i}} & \, & \downarrow\alpha_{X} & \, & \downarrow\alpha_{Y}\\
\mathcal{O}_{F}^{\mathrm{dim}Z_{i}} & \stackrel{j}{\hookrightarrow} & \mathcal{O}_{F}^{\mathrm{dim}X} & \stackrel{\psi}{\longrightarrow} & \mathcal{O}_{F}^{\mathrm{dim}Y},
\end{array}
\]
where $j:\mathcal{O}_{F}^{\mathrm{dim}Z_{i}}\rightarrow\mathcal{O}_{F}^{\mathrm{dim}X}$
is the inclusion to the first $\mathrm{dim}Z_{i}$ coordinates. After
an analytic change of coordinates we may assume that:
\[
\mathrm{ker}d\psi_{z}=\mathrm{span}\{e_{1},...,e_{\mathrm{dim}X-\mathrm{dim}Y+r}\},
\]
 for any $z\in\mathcal{O}_{F}^{dimZ_{i}}$. By Lemma \ref{Lemma 3.3},
we have that $\mathcal{M}:=\alpha_{Y}(\mathcal{N})$ is a local basis
of bounded eccentricity at $0\in\mathcal{O}_{F}^{\mathrm{dim}Y}$.
Note that $\mu:=(\alpha_{X})_{*}(1_{U_{X}}\cdot m)$ is a non-negative
Schwartz measure that does not vanish at $0$, and that $\psi_{*}(\mu)$
is $\mathcal{M}$-bounded. By Proposition \ref{Prop 2.6}, after restricting
to a small enough ball around $0$ and applying a homothety, we can
assume that $\mu$ is the normalized Haar measure. 

As part of the data, for any $M_{j}\in\mathcal{M}$ we are given by
$B_{max_{j}}(0)$ and $B_{min_{j}}(0)$, and there exists $\delta,C>0$
such that for any $M_{j}\in\mathcal{M}_{\delta}:=\{M_{j}\in\mathcal{M}|M_{j}\subseteq B_{\delta}(0)\}$,
we have $B_{max_{j}}(0)\subseteq M_{j}\subseteq B_{min_{j}}(0)$ and
$\frac{\lambda(B_{min_{j}})}{\lambda(B_{max_{j}})}\leq C$. For any
$0<\epsilon<1$, set
\[
A_{\epsilon}:=\left\{ (x_{1},...,x_{\mathrm{dim}X})\in\mathcal{O}_{F}^{\mathrm{dim}X}|\left|x_{k}\right|<\epsilon^{n_{k}}\right\} ,
\]
where $n_{k}=0$ if $1\leq k\leq\mathrm{dim}Z_{i}$; $n_{k}=1$ for
$\mathrm{dim}Z_{i}+1\leq k\leq\mathrm{dim}X-\mathrm{dim}Y+r$; and
$n_{k}=2$ for $\mathrm{dim}X-\mathrm{dim}Y+r+1\leq k\leq\mathrm{dim}X$. 

By choosing $\delta$ small enough, we may find a constant $D>0$
such that $\psi(A_{D\sqrt{\epsilon}})\subseteq B_{\epsilon}(0)$ for
every $\epsilon<\delta$. In particular, for any $M_{j}\in\mathcal{M}_{\delta}$
we get that $\psi(A_{D\cdot\sqrt{max_{j}}})\subseteq B_{max_{j}}(0)$
so $\psi^{-1}(B_{max_{j}}(0))\supseteq A_{D\cdot\sqrt{max_{j}}}$.
Denote $\epsilon_{j}:=\sqrt{max_{j}}$ and notice that there exists
a constant $L>0$ such that for any $j$ with $M_{j}\in\mathcal{M}_{\delta}$,
it holds that: 
\begin{eqnarray*}
\mu(A_{D\epsilon_{j}}) & \geq & L\cdot\left(D\epsilon_{j}\right)^{\mathrm{dim}X-\mathrm{dim}Y+r-\mathrm{dim}Z_{i}+2(\mathrm{dim}Y-r)}\\
 & = & D'\cdot\epsilon_{j}^{\mathrm{dim}X+\mathrm{dim}Y-r-\mathrm{dim}Z_{i}}\\
 & \geq & D'\cdot\epsilon_{j}^{2\mathrm{dim}Y-r},
\end{eqnarray*}

where $D'$ is some positive constant. Altogether, we have: 
\begin{eqnarray*}
\frac{\psi_{*}(\mu)(M_{j})}{\lambda(M_{j})} & \geq & \frac{\psi_{*}(\mu)(B_{max_{j}}(0))}{\lambda(B_{min_{j}}(0))}\geq\frac{1}{C}\frac{\psi_{*}(\mu)(B_{max_{j}}(0))}{\lambda(B_{max_{j}}(0))}\\
 & \geq & \frac{1}{C}\frac{\mu(A_{D\epsilon_{j}})}{\lambda(B_{max_{j}}(0))}\geq\frac{D'}{C}\frac{\epsilon_{j}^{2dimY-r}}{\epsilon_{j}^{2dimY}}\geq\frac{D'}{C}\epsilon_{j}^{-r}.
\end{eqnarray*}
Since $\mathcal{M}_{\delta}$ is a local basis, the above equation
is true for arbitrary small $\epsilon_{j}$, so we have a contradiction
to the $\mathcal{M}$-boundedness of $\psi_{*}(\mu)$.
\end{proof}
\begin{cor}
We have that $\varphi$ is flat at $x$, and that there is a Zariski
neighborhood $U_{0}$ of $x$ such that $Z\cap U_{0}$ is reduced
and a local complete intersection $(LCI)$.
\end{cor}

\begin{proof}
Let $Z_{1},...,Z_{n}$ be the absolutely irreducible components of
$Z$ containing $x$. By the last lemma, each $Z_{i}$ contains a
smooth point of $\varphi$, so $\mathrm{dim}_{x}Z:=\underset{i}{\mathrm{max}}\mathrm{dim}Z_{i}=\mathrm{dim}X-\mathrm{dim}Y$.
Hence, we may find a neighborhood $U_{0}$ of $x$ such that $\varphi|_{U_{0}}$
is flat over $\varphi(x)$ (and in particular flat at $x$). As a
consequence, we get that $Z\cap U_{0}$ is an $(LCI)$, and in particular
Cohen-Macaulay. Since $Z\cap X^{S}\cap U_{0}$ is a dense in $Z\cap U_{0}$
and $Z\cap X^{S}=Z^{sm}$ (see e.g \cite[III.10.2]{Har77}) it follows
that $Z\cap U_{0}$ is generically reduced. Since $Z\cap U_{0}$ is
also Cohen-Macaulay, it now follows from (e.g \cite[Exercise 26.3.B]{Vak})
that it is reduced.
\end{proof}
Without loss of generality, we assume $X=U_{0}$. The following lemma
implies that $\varphi$ is $(FRS)$ at $x$, and thus finishes the
proof of Theorem \ref{Main auxilary theorem}:
\begin{lem}
$x$ is a rational singularity of $Z$. 
\end{lem}

\begin{proof}
After further restricting to Zariski open neighborhoods of $x$ and
$\varphi(x)$, we may assume that $X$ and $Y$ are affine, with $\Omega_{X}^{top},\Omega_{Y}^{top}$
free. Fix invertible top forms $\omega_{X}\in\Omega_{X}^{top}[X]$,
$\omega_{Y}\in\Omega_{Y}^{top}[Y]$. We may find an invertible section
$\eta\in\Omega_{Z}^{top}[Z]$, such that $\eta|_{Z^{sm}}=\frac{\omega_{X}|_{X^{S}}}{\text{\ensuremath{\varphi}*}(\omega_{Y})}$
(for more details see the last part of the proof of \cite[Theorem 3.4]{AA16}).
We denote $\omega_{Z}:=\eta|_{Z^{sm}}$. 

Fix a finite extension $k'/k$. By assumption, there exists a non-Archimedean
local field $F\supseteq k'$ and a non-negative Schwartz measure $m$
on $X(F)$ that does not vanish at $x$, such that $\varphi_{*}(m)$
is bounded with respect to a local basis $\mathcal{N}$ of bounded
eccentricity. Write $m$ as $m=f\cdot\left|\omega_{X}\right|_{F}$.
Since $Z$ is an $(LCI)$, it is also Gorenstein, so by \cite[Corollary 3.15]{AA16},
it is enough to prove that $\int_{X^{S}\cap Z(F)}f\left|\omega_{Z}\right|_{F}<\infty$
for any such $k'/k$ and $F$.

Fix some embedding of $X$ into an affine space, and let $d$ be the
metric on $X(F)$ induced from the valuation metric. Define a function
$h_{\epsilon}:X(F)\rightarrow\mathbb{R}$ by $h_{\epsilon}(x')=1$
if $d(x',\left(X^{S}(F)\right)^{C})\geq\epsilon$ and $h_{\epsilon}(x')=0$
otherwise. Notice that $h_{\epsilon}$ is smooth, and $f\cdot h_{\epsilon}$
is a Schwartz function whose support lies in $X^{S}(F)$.

Using Proposition\textcolor{red}{{} }\ref{Prop 2.7}, we have $\varphi_{*}(f\cdot h_{\epsilon}\left|\omega_{X}\right|_{F})=g_{\epsilon}\left|\omega_{Y}\right|_{F}$,
where $g_{\epsilon}(\varphi(x))=\int_{X^{S}\cap Z(F)}f\cdot h_{\epsilon}\left|\omega_{Z}\right|_{F}$.
Note that $f$ is non-negative and $f\cdot h_{\epsilon}$ is monotonically
increasing when $\epsilon\rightarrow0$, and converges pointwise to
$f$. By Lebesgue's monotone convergence theorem we have:
\[
\int_{X^{S}\cap Z(F)}f\left|\omega_{Z}\right|_{F}=\underset{\epsilon\rightarrow0}{\mathrm{lim}}\int_{X^{S}\cap Z(F)}fh_{\epsilon}\left|\omega_{Z}\right|_{F}=\underset{\epsilon\rightarrow0}{\mathrm{lim}}g_{\epsilon}(\varphi(x)).
\]
It is left to show that $g_{\epsilon}(\varphi(x))$ is bounded in
$\epsilon$ and we are done. By our assumption, $\varphi_{*}(f\cdot\left|\omega_{X}\right|_{F})$
is $\mathcal{N}$-bounded, so there exists $\delta>0$ and $M>0$
such that for all $N_{i}\in\mathcal{N}_{\delta}$, 
\[
\underset{i}{\mathrm{sup}}\frac{\varphi_{*}(f\left|\omega_{X}\right|_{F})(N_{i})}{\left|\omega_{Y}\right|_{F}(N_{i})}<M.
\]
Note that we used the fact that for small enough $\delta$, $\left|\omega_{Y}\right|_{F}$
is just the normalized Haar measure up to homothety. Finally, we obtain:

\begin{eqnarray*}
\int_{X^{S}\cap Z(F)}f\left|\omega_{Z}\right|_{F} & = & \underset{\epsilon\rightarrow0}{\mathrm{lim}}g_{\epsilon}(\varphi(x))\\
 & = & \underset{\epsilon\rightarrow0}{\mathrm{lim}}\left(\underset{i\rightarrow\infty}{\mathrm{lim}}\frac{\varphi_{*}(f\cdot h_{\epsilon}\left|\omega_{X}\right|_{F})(N_{i})}{\left|\omega_{Y}\right|_{F}(N_{i})}\right)\\
 & \leq & \left(\underset{i}{\mathrm{sup}}\frac{\varphi_{*}(f\left|\omega_{X}\right|_{F})(N_{i})}{\left|\omega_{Y}\right|_{F}(N_{i})}\right)<M.
\end{eqnarray*}
\end{proof}

\section{Proof of the main theorem}

For any prime power $q=p^{r}$, we denote the unique unramified extension
of $\mathbb{Q}_{p}$ of degree $r$ by $\mathbb{Q}_{q}$, its ring
of integers by $\mathbb{Z}_{q}$, and the maximal ideal of $\mathbb{Z}_{q}$
by $\mathfrak{m}_{q}$. Recall that for a finite type $\mathbb{Z}$-scheme
$X$ and a finite ring $A$, we have defined $h_{X}(A):=\frac{\left|X(A)\right|}{\left|A\right|^{\mathrm{dim}X_{\mathbb{Q}}}}$.
In this section we prove the following slightly stronger version of
Theorem \ref{Main theorem}: 
\begin{thm}
\label{Main extended theorem} Let $X$ be a scheme of finite type
over $\mathbb{Z}$ such that $X_{\mathbb{Q}}$ is equi-dimensional
and a local complete intersection. Then the following conditions are
equivalent:

i) For any $n\in\mathbb{N}$, $\underset{p\rightarrow\infty}{\mathrm{lim}}h_{X}(\mathbb{Z}/p^{n}\mathbb{Z})=1$.

ii) There is a finite set $S$ of prime numbers and a constant $C$,
such that $\left|h_{X}(\mathbb{Z}/p^{n}\mathbb{Z})-1\right|<Cp^{-1/2}$
for any prime $p\notin S$ and any $n\in\mathbb{N}$.

iii) $X_{\mathbb{\overline{Q}}}$ is reduced, irreducible and has
rational singularities. 

iv) $X_{\overline{\mathbb{Q}}}$ is irreducible and there exists $C>0$
such that $h_{X}(\mathbb{Z}/p^{n}\mathbb{Z})<C$ for any prime $p$
and $n\in\mathbb{N}$.

iv') $X_{\overline{\mathbb{Q}}}$ is irreducible and for any prime
power $q$, the sequence $n\mapsto h_{X}(\mathbb{Z}_{q}/\mathfrak{m}_{q}^{n})$
is bounded.

v) $X_{\overline{\mathbb{Q}}}$ is irreducible and there exists a
finite set $S$ of primes, such that for any $p\notin S$ , the sequence
$n\mapsto h_{X}(\mathbb{Z}/p^{n}\mathbb{Z})$ is bounded.

Moreover, $iii),iv),iv')$ and $v)$ are equivalent without demanding
that $X_{\overline{\mathbb{Q}}}$ is irreducible. 
\end{thm}

We divide the proof of the theorem to two main parts that corresponds
to the implication $v)\Longrightarrow iii)$ (Section \ref{subsec:Boundedness-implies-rational})
and the implication $iii)\Longrightarrow iv')$ (Section \ref{subsec:Rational-singularities-implies}).
The equivalence of condition $i),ii)$ and $iii)$ is proved in \cite[Theorem 3.0.3]{AA}
(see Theorem \ref{Raminir theorem}). The implications $iv)\Longrightarrow v)$
and $iv')\Longrightarrow v)$ are trivial and $iv')\Longrightarrow iv)$
follows by applying $q=p$ to $iv)$ and from $ii)$. For the proof
we need:
\begin{lem}
\label{Lemma 4.2}(e.g \cite[Proposition 3.2.1]{AA}) Let $X=U_{1}\cup U_{2}$
be an open cover of a scheme. Then for any finite local ring $A$,
we have:

1) $\left|X(A)\right|=\left|U_{1}(A)\right|+\left|U_{2}(A)\right|-\left|U_{1}\cap U_{2}(A)\right|$. 

2) $\left|X(A)\right|\geq\left|U_{1}(A)\right|$.
\end{lem}

The following lemma is a consequence of Chebotarev's density theorem
and Hensel's lemma.
\begin{lem}
\label{Lemma 4.3}(\cite[Lemma 3.15]{GH}) Let $X$ be a finite type
$\mathbb{Z}$-scheme and let $x\in X(\overline{\mathbb{Q}})$. Then:

1) There exists a finite extension $k$ of $\mathbb{Q}$, such that
$x\in X(k)$.

2) For any finite extension $k/\mathbb{Q}$ as in 1), there exist
infinitely many primes $p$ with $i_{p}:k\hookrightarrow\mathbb{Q}_{p}$
such that $i_{p*}(x)\in X(\mathbb{Z}_{p})$, where $i_{p*}:X(k)\hookrightarrow X(\mathbb{Q}_{p})$.
\end{lem}

\subsection{\label{subsec:Boundedness-implies-rational} Boundedness implies
rational singularities }
\begin{thm}
\label{Theorem 4.4} Let $X$ be a finite type $\mathbb{Z}$-scheme
such that $X_{\mathbb{Q}}$ is a local complete intersection. Assume
that there exists a finite set of primes $S$, such that for any $p\notin S$,
the sequence $n\mapsto h_{X}(\mathbb{Z}/p^{n}\mathbb{Z})$ is bounded.
Then $X_{\mathbb{\overline{Q}}}$ is reduced and has rational singularities. 
\end{thm}

\begin{proof}
\textbf{Step 1: }Reduction to the case when $X_{\mathbb{Q}}$ is a
complete intersection in an affine space $(CIA)$. 

Let $\bigcup_{i=1}^{l}\overline{X}_{i}$ be an affine cover of $X_{\mathbb{Q}}$,
with each $\overline{X}_{i}$ a $(CIA)$. For any $i$, there is a
finite set $S_{i}$ of primes, such that $\overline{X}_{i}$ is defined
over $\mathbb{Z}[S_{i}^{-1}]$ and thus it has a finite type $\mathbb{Z}$-model,
denoted $X_{i}$. By Lemma \ref{Lemma 4.2}, for each $p\notin S_{i}$
we have $\left|X_{i}(\mathbb{Z}/p^{n}\mathbb{Z})\right|\leq\left|X(\mathbb{Z}/p^{n}\mathbb{Z})\right|$
and thus $n\mapsto h_{X_{i}}(\mathbb{Z}/p^{n}\mathbb{Z})$ is bounded
for each $p\notin S_{i}\cup S$. By our assumption, this implies that
each $\left(X_{i}\right)_{\mathbb{\overline{Q}}}$ is reduced and
has rational singularities, and thus also $X_{\mathbb{\overline{Q}}}$. 

\textbf{Step 2:} Proof for the case when $X_{\mathbb{Q}}$ is a $(CIA)$.

By Proposition \ref{Prop 2.3} we have an inclusion $\overline{\psi}:X_{\mathbb{Q}}\hookrightarrow\mathbb{A}_{\mathbb{Q}}^{M}$
and a morphism $\overline{\varphi}:\mathbb{A}_{\mathbb{Q}}^{M}\longrightarrow\mathbb{A}_{\mathbb{Q}}^{N}$,
flat over $0$, such that $\overline{\psi}:X_{\mathbb{Q}}\simeq\overline{\varphi}^{-1}(0)$.
As in Step 1, there exists a set $S_{1}$ of primes, and morphisms
$\varphi:\mathbb{A}_{\mathbb{Z}[S_{1}^{-1}]}^{M}\longrightarrow\mathbb{A}_{\mathbb{Z}[S_{1}^{-1}]}^{N}$
and $\psi:X_{\mathbb{Z}[S_{1}^{-1}]}\hookrightarrow\mathbb{A}_{\mathbb{Z}[S_{1}^{-1}]}^{M}$,
such that $\varphi_{\mathbb{Q}}=\overline{\varphi}$, $\psi_{\mathbb{Q}}=\overline{\psi}$,
$\varphi$ is flat over $0$, and $\psi:X_{\mathbb{Z}[S_{1}^{-1}]}\simeq\varphi^{-1}(0)$. 

It is enough to prove that for any finite extension $k/\mathbb{Q}$
and any $y\in\left(\varphi^{-1}(0)\right)(k)$, the map $\varphi_{k}:\mathbb{A}_{k}^{M}\longrightarrow\mathbb{A}_{k}^{N}$
is $(FRS)$ at $y$. 

Fix $y\in\left(\varphi^{-1}(0)\right)(k)$ and let $k'$ be a finite
extension of $k$. By Lemma \ref{Lemma 4.3}, there exists an infinite
set of primes $T$ such that for any $p\in T$ we have an inclusion
$i_{p}:k'\hookrightarrow\mathbb{Q}_{p}$ and $i_{p*}(y)\in\mathbb{Z}_{p}^{M}$.
Choose $p\in T\backslash(S\cup S_{1})$ and consider the local basis
of balls $\{p^{n}\mathbb{Z}_{p}^{N}\}_{n}$ at $0$, which clearly
has bounded eccentricity. Let $\mu$ be the normalized Haar measure
on $\mathbb{Z}_{p}^{M}$ and notice that $\mu$ does not vanish at
$y$. By Theorem \ref{Main auxilary theorem}, in order to prove that
$\varphi_{k}:\mathbb{A}_{k}^{M}\longrightarrow\mathbb{A}_{k}^{N}$
is $(FRS)$ at $y$ it is enough to show that the sequence 
\[
n\mapsto\frac{\left(\left(\varphi_{\mathbb{Z}_{p}}\right)_{*}\mu\right)(p^{n}\mathbb{Z}_{p}^{N})}{\lambda(p^{n}\mathbb{Z}_{p}^{N})}
\]
is bounded (for any $k'$ and $p$ as above), where $\lambda$ is
the normalized Haar measure on $\mathbb{Q}_{p}^{N}$. Consider $\pi_{N,n}:\mathbb{Z}_{p}^{N}\longrightarrow(\mathbb{Z}/p^{n}\mathbb{Z})^{N}$
and notice that the following diagram is commutative: 
\[
\begin{array}{ccc}
\mathbb{Z}_{p}^{M} & \overset{\varphi_{\mathbb{Z}_{p}}}{\longrightarrow} & \mathbb{Z}_{p}^{N}\\
\downarrow\pi_{M,n} & \, & \downarrow\pi_{N,n}\\
(\mathbb{Z}/p^{n}\mathbb{Z})^{M} & \overset{\varphi_{\mathbb{Z}/p^{n}}}{\longrightarrow} & (\mathbb{Z}/p^{n}\mathbb{Z})^{N}.
\end{array}
\]
Therefore we have 
\begin{align*}
\mu(\varphi_{\mathbb{Z}_{p}}^{-1}(p^{n}\mathbb{Z}_{p}^{N})) & =\mu(\varphi_{\mathbb{Z}_{p}}^{-1}\circ\pi_{N,n}^{-1}(0))=\mu(\pi_{M,n}^{-1}\circ\varphi_{\mathbb{Z}/p^{n}}^{-1}(0))\\
 & =p^{-Mn}\cdot\left|\varphi_{\mathbb{Z}/p^{n}}^{-1}(0)\right|=p^{-Mn}\cdot\left|X(\mathbb{Z}/p^{n}\mathbb{Z})\right|,
\end{align*}

and hence
\[
\frac{\left(\left(\varphi_{\mathbb{Z}_{p}}\right)_{*}\mu\right)(p^{n}\mathbb{Z}_{p}^{N})}{\lambda(p^{n}\mathbb{Z}_{p}^{N})}=\frac{\left|X(\mathbb{Z}/p^{n}\mathbb{Z})\right|}{p^{(M-N)\cdot n}}=h_{X}(\mathbb{Z}/p^{n}\mathbb{Z})
\]
 is bounded and we are done. 
\end{proof}

\subsection{\label{subsec:Rational-singularities-implies} Rational singularities
implies boundedness}

In the last section we proved the implication $v)\Longrightarrow iii)$
of Theorem \ref{Main extended theorem}. In this subsection we prove
that $iii)$ implies $iv')$. We divide the proof to three cases: 
\begin{enumerate}
\item $X$ is a $(CIA)$.
\item $X_{\mathbb{Q}}$ is a $(CIA)$.
\item $X_{\mathbb{Q}}$ is an $(LCI)$.
\end{enumerate}

\subsubsection{Proof for the case that $X$ is a $(CIA)$}

As stated in \cite[Remark 3.04]{AA}, this case can be proved using
arguments similar to those in the proof of \cite[Theorem 3.0.3]{AA},
although the details were omitted. For completeness, we present a
proof.
\begin{prop}
\label{Proposition 4.5} If $X$ is a $(CIA)$, then $iii)\Longrightarrow iv')$.
\end{prop}

\begin{proof}
By Proposition \ref{prop 2.4}, there exists an inclusion $X\hookrightarrow\mathbb{A}_{\mathbb{Z}}^{M}$
and a morphism $\varphi:\mathbb{A}_{\mathbb{Z}}^{M}\longrightarrow\mathbb{A}_{\mathbb{Z}}^{N}$,
flat over $0$, such that $X\simeq\varphi^{-1}(0)$. Consider $\varphi_{\mathbb{Q}}:\mathbb{A}_{\mathbb{Q}}^{M}\longrightarrow\mathbb{A}_{\mathbb{Q}}^{N}$
and notice that $\varphi_{\mathbb{Q}}$ is $(FRS)$ at any $x\in\varphi_{\mathbb{Q}}^{-1}(0)(\mathbb{\overline{Q}})$,
as $X_{\mathbb{\overline{Q}}}$ has rational singularities.

Let $\mu$ be the normalized Haar measure on $\mathbb{Z}_{q}^{M}$.
As in the proof of Step 2 of Theorem \ref{Theorem 4.4}, we have the
following commutative diagram:
\[
\begin{array}{ccc}
\mathbb{Z}_{q}^{M} & \overset{\varphi_{\mathbb{Z}_{q}}}{\longrightarrow} & \mathbb{Z}_{q}^{N}\\
\downarrow\pi_{n,M} & \, & \downarrow\pi_{n,N}\\
(\mathbb{Z}_{q}/\mathfrak{m}_{q}^{n})^{M} & \overset{\varphi_{\mathbb{Z}_{q}/\mathfrak{m}_{q}^{n}}}{\longrightarrow} & (\mathbb{Z}_{q}/\mathfrak{m}_{q}^{n})^{N}.
\end{array}
\]
In order to show that $h_{X}(\mathbb{Z}_{q}/\mathfrak{m}_{q}^{n})$
is bounded, it is enough to show that $(\varphi_{\mathbb{Z}_{q}})_{*}\mu$
has bounded density with respect to the local basis $\{p^{n}\mathbb{Z}_{q}^{N}\}_{n}$. 

After base change to $\mathbb{Q}_{q}$, we have a map $\varphi_{\mathbb{Q}_{q}}:\mathbb{A}_{\mathbb{Q}_{q}}^{M}\longrightarrow\mathbb{A}_{\mathbb{Q}_{q}}^{N}$,
which is $(FRS)$ at any point $x\in X(\mathbb{Q}_{q})$.

For any $t\in\mathbb{N}$, consider the set $U_{t}=\varphi_{\mathbb{Z}_{q}}^{-1}(p^{t}\mathbb{Z}_{q}^{N})$
and note that it is open, closed and compact. We claim that there
exists $R\in\mathbb{N}$, such that for any $t>R$ we have that $\varphi$
is $(FRS)$ at any point $y\in U_{t}$. Indeed, otherwise we may construct
a sequence $x_{t}\in U_{t}$ such that $\varphi$ is not $(FRS)$
at $x_{t}$. By a theorem of Elkik (\cite{Elk78}, \cite[Theorem 6.3]{AA16}),
the $(FRS)$ locus of $\varphi$ is an open set. After choosing a
convergent subsequence $\{x_{t_{j}}\}$, we obtain that $\varphi_{\mathbb{Q}_{q}}$
is not $(FRS)$ at the limit $x_{0}\in\mathbb{Z}_{q}^{M}$. But $\varphi_{\mathbb{Q}_{q}}(x_{0})\in\cap_{t}\varphi_{\mathbb{Q}_{q}}(U_{t})=\{0\}$
so $x_{0}\in X(\mathbb{Q}_{q})$ and we get a contradiction.

Finally, by Theorem \ref{Auxilery Raminir theorem}, the measure $(\varphi_{\mathbb{Z}_{q}})_{*}\mu|_{U_{R}}$
has continuous density, and in particular bounded with respect to
the local basis $\{p^{n}\mathbb{Z}_{q}^{N}\}_{n}$. Hence, from the
definition of $U_{R}$, we have for $n>R$: 
\[
h_{X}(\mathbb{Z}_{q}/\mathfrak{m}_{q}^{n})=\frac{(\varphi_{\mathbb{Z}_{q}})_{*}\mu(p^{n}\mathbb{Z}_{q}^{N})}{q^{-nN}}=\frac{(\varphi_{\mathbb{Z}_{q}})_{*}\mu|_{U_{R}}(p^{n}\mathbb{Z}_{q}^{N})}{q^{-nN}}<C,
\]
for some constant $C>0$ and we are done. 
\end{proof}

\subsubsection{\label{subsec:Some-constructions} Some constructions}

Let $X$ be an affine $\mathbb{Z}$-scheme with a coordinate ring
\[
\mathbb{Z}[X]:=\mathbb{Z}[x_{1},...,x_{c}]/(f_{1},...,f_{m}),
\]
and fix $K\in\mathbb{N}$. 
\begin{enumerate}
\item For any $g\in\mathbb{Z}[x_{1},...,x_{c}]$ denote by $g_{K}\in\mathbb{Q}[x_{1},...,x_{c}]$
the function $g_{K}(x_{1},...,x_{c}):=g(\frac{x_{1}}{K},...,\frac{x_{l}}{K})$. 
\item For any $\varphi:\mathbb{A}_{\mathbb{Z}}^{M}\rightarrow\mathbb{A}_{\mathbb{Z}}^{N}$
of the form $\varphi=(\varphi_{1},...,\varphi_{N})$, we denote by
$\varphi_{K}:\mathbb{A}_{\mathbb{Q}}^{M}\rightarrow\mathbb{A}_{\mathbb{Q}}^{N}$
the morphism $\varphi_{K}:=((\varphi_{1})_{K},...,(\varphi_{N})_{K})$.
\item Let $r(K)\in\mathbb{N}$ be minimal such that $K^{r(K)}(f_{i})_{K}$
has integer coefficients for any $i$. Denote by $\widetilde{X}_{K}$
the $\mathbb{Z}$-scheme with the following coordinate ring:
\[
\mathbb{Z}[\widetilde{X}_{K}]:=\mathbb{Z}[x_{1},...,x_{c}]/(K^{r(K)}(f_{1})_{K},...,K^{r(K)}(f_{m})_{K}).
\]
\item For any $\mathbb{Q}$-morphism $\psi:X_{\mathbb{Q}}\rightarrow\mathbb{A}_{\mathbb{Q}}^{M}$
of the form $\psi=(\psi_{1},...,\psi_{N})$ denote by $K\psi:=(K\cdot\psi_{1},...,K\cdot\psi_{N})$. 
\item For any affine $\mathbb{Q}$-scheme $Z$, with $\mathbb{Q}[Z]=\mathbb{Q}[y_{1},...,y_{d}]/(g_{1},...,g_{k})$
and a $\mathbb{Q}$-morphism $\phi:Z\rightarrow X_{\mathbb{Q}}$,
we may define a morphism $K\phi:Z\rightarrow\left(\widetilde{X}_{K}\right)_{\mathbb{Q}}$
by $K\phi(y_{1},...,y_{d}):=K\cdot\phi(y_{1},...,y_{d})$.
\end{enumerate}

\subsubsection{Proof for the case that $X_{\mathbb{Q}}$ is a $(CIA)$.}

In this case, we have an inclusion $\psi:X_{\mathbb{Q}}\hookrightarrow\mathbb{A}_{\mathbb{Q}}^{M}$
and a morphism $\varphi:\mathbb{A}_{\mathbb{Q}}^{M}\longrightarrow\mathbb{A}_{\mathbb{Q}}^{N}$,
flat over $0$, such that $X_{\mathbb{Q}}\simeq\varphi^{-1}(0)$. 
\begin{lem}
\label{Lemma 4.6} Let $X$ be a finite type $\mathbb{Z}$-scheme,
such that $X_{\mathbb{Q}}$ is a $(CIA)$, defined by the morphisms
$\varphi,\psi$ as above. Then there exists a $\mathbb{Z}$-scheme
$\widehat{X}_{\varphi,\psi}$, which is a $(CIA)$, and a $\mathbb{Z}$-morphism
$\phi:X\rightarrow\widehat{X}_{\varphi,\psi},$ such that $\phi_{\mathbb{Q}}$
is an isomorphism. 
\end{lem}

\begin{proof}
Let $\mathbb{Z}[X]:=\mathbb{Z}[x_{1},...,x_{c}]/(f_{1},...,f_{m})$
be the coordinate ring of $X$. Denote by $S=\{p_{1},...,p_{s}\}$
the set of all prime numbers that appear in the denominators of the
polynomial maps $\psi$ and $\varphi$, and set $P':=\underset{p_{i}\in S}{\prod}p_{i}$.
Let $t\in\mathbb{N}$ be minimal such that $\left(P'\right)^{t}\psi$
has integer coefficients. Denote $P:=\left(P'\right)^{t}$ and notice
that $P\psi$ is a $\mathbb{Z}$-morphism. Let $\varphi_{P}:\mathbb{A}_{\mathbb{Q}}^{M}\rightarrow\mathbb{A}_{\mathbb{Q}}^{N}$
be as defined in \ref{subsec:Some-constructions}. Notice that there
exists $m\in\mathbb{N}$ such that $P^{m}\varphi_{P}$ has coefficients
in $\mathbb{Z}$. We now have the following $\mathbb{Z}$-morphisms
$X\overset{P\psi}{\longrightarrow}\mathbb{A}_{\mathbb{Z}}^{M}\overset{P^{m}\varphi_{P}}{\longrightarrow}\mathbb{A}_{\mathbb{Z}}^{N}$.
Set $\widehat{X}_{\varphi,\psi}$ to be the fiber $(P^{m}\varphi_{P})^{-1}(0)$
and notice that $\phi:=P\psi$ is a $\mathbb{Z}$-morphism from $X$
to $\widehat{X}_{\varphi,\psi}$, such that $\phi_{\mathbb{Q}}$ is
an isomorphism, and $\widehat{X}_{\varphi,\psi}$ is a $(CIA)$.
\end{proof}
\begin{lem}
\label{Lemma 4.7} Let $X$ and $Y$ be affine $\mathbb{Z}$-schemes
and $\phi:X\longrightarrow Y$ be a $\mathbb{Z}$-morphism, such that
$\phi_{\mathbb{Q}}$ is an isomorphism. Then there exist $c,N\in\mathbb{N}$,
such that for any prime power $q$ and any $n$:
\[
\left|X(\mathbb{Z}_{q}/\mathfrak{m}_{q}^{n})\right|\leq q^{N\cdot c}\cdot\left|Y(\mathbb{Z}_{q}/\mathfrak{m}_{q}^{n})\right|.
\]
\end{lem}

\begin{proof}
$\phi$ induces a map $\phi_{n}:X(\mathbb{Z}_{q}/\mathfrak{m}_{q}^{n})\rightarrow Y(\mathbb{Z}_{q}/\mathfrak{m}_{q}^{n})$.
It is enough to show that $\phi_{n}$ has fibers of size at most $q^{N\cdot c}$.
Assume that $\mathbb{Z}[X]=\mathbb{Z}[x_{1},...,x_{c}]/(f_{1},...,f_{m})$.
As in \ref{subsec:Some-constructions}, we may choose $K,r(K)\in\mathbb{N}$
such that $\widetilde{X}_{K}$ is a $\mathbb{Z}$-scheme with a coordinate
ring 
\[
\mathbb{Z}[\widetilde{X}_{K}]:=\mathbb{Z}[x_{1},...,x_{c}]/(K^{r(K)}(f_{1})_{K},...,K^{r(K)}(f_{m})_{K}),
\]

and $K\phi^{-1}:Y\longrightarrow\widetilde{X}_{K}$ is a $\mathbb{Z}$-morphism.
The map $(K\phi^{-1}\circ\phi):X\longrightarrow\widetilde{X}_{K}$
is just coordinate-wise multiplication by $K$. Thus $\left(K\phi^{-1}\right)_{n}\circ\phi_{n}:X(\mathbb{Z}_{q}/\mathfrak{m}_{q}^{n})\longrightarrow\widetilde{X}_{K}(\mathbb{Z}_{q}/\mathfrak{m}_{q}^{n})$
sends $(a_{1},...,a_{c})\in X(\mathbb{Z}_{q}/\mathfrak{m}_{q}^{n})$
to $(Ka_{1},...,Ka_{c})\in\widetilde{X}_{K}(\mathbb{Z}_{q}/\mathfrak{m}_{q}^{n})$. 

For any prime $p$, let $N(p)$ be the maximal integer such that $p^{N(p)}|K$.
Note that the map $(a_{1},...,a_{n})\mapsto(Ka_{1},...,Ka_{n})$ from
$\left(\mathbb{Z}_{q}/\mathfrak{m}_{q}^{n}\right)^{c}$ to $\left(\mathbb{Z}_{q}/\mathfrak{m}_{q}^{n}\right)^{c}$
has fibers of size $q^{N(p)\cdot c}$ for $n>N(p)$. Indeed, for $(b_{1},...,b_{c})\in\left(\mathbb{Z}_{q}/\mathfrak{m}_{q}^{n}\right)^{c}$,
$(Ka_{1},...,Ka_{c})=(b_{1},...,b_{c})$ if and only if $Ka_{i}=b_{i}$
for any $1\leq i\leq c$. Since $\frac{K}{p^{N(p)}}$ is invertible
in $\mathbb{Z}_{q}/\mathfrak{m}_{q}^{n}$, it is equivalent to demand
that $p^{N(p)}a_{i}=c_{i}$ for some multiple $c_{i}$ of $b_{i}$
by an invertible element. Hence, we can reduce to the case of the
map $(a_{1},...,a_{n})\mapsto(p^{N(p)}a_{1},...,p^{N(p)}a_{n})$,
which clearly has fibers of size $q^{N(p)\cdot c}$ for $n>N(p)$. 

Note that for any $y\in Y(\mathbb{Z}_{q}/\mathfrak{m}_{q}^{n})$ we
have $\left|\phi_{n}^{-1}(y)\right|\leq\left|\left(\left(K\phi^{-1}\right)_{n}\circ\phi_{n}\right)^{-1}(x)\right|$,
where $x=\left(K\phi^{-1}\right)_{n}(y)$. Since the fibers of $\left(K\phi^{-1}\right)_{n}\circ\phi_{n}$
are of size bounded by $q^{N(p)c}$, so does the fibers of $\phi_{n}$.
We may take $N:=K>N(p)$ and we are done. 
\end{proof}
\begin{cor}
Let $X$ be a finite type $\mathbb{Z}$-scheme such that $X_{\mathbb{Q}}$
is a $(CIA)$. Then condition $iii)$ of Theorem \ref{Main extended theorem}
implies condition $iv')$.
\end{cor}

\begin{proof}
By Lemma \ref{Lemma 4.6}, we may choose a $\mathbb{Z}$-scheme $\widehat{X}$,
which is a $(CIA)$, and a $\mathbb{Z}$-morphism $\phi:X\longrightarrow\widehat{X},$
such that $\phi_{\mathbb{Q}}$ is an isomorphism. By Proposition \ref{Proposition 4.5}
and Lemma \ref{Lemma 4.7}, there exists $c,N\in\mathbb{N}$, such
that for any prime power $q$, there exists $C>0$ such that:
\[
h_{X}(\mathbb{Z}_{q}/\mathfrak{m}_{q}^{n})=\frac{\left|X(\mathbb{Z}_{q}/\mathfrak{m}_{q}^{n})\right|}{q^{n\mathrm{dim}X_{\mathbb{Q}}}}\leq q^{c\cdot N}\cdot\frac{\left|\widehat{X}(\mathbb{Z}_{q}/\mathfrak{m}_{q}^{n})\right|}{q^{n\mathrm{dim}X_{\mathbb{Q}}}}\leq q^{c\cdot N}\cdot C,
\]
and hence condition $iv')$ holds.
\end{proof}

\subsubsection{Proof for the case when $X_{\mathbb{Q}}$ is an $(LCI)$.}

Using Lemma \ref{Lemma 4.2}, we may reduce to the case when $X$
is affine, with coordinate ring $\mathbb{Z}[X]:=\mathbb{Z}[x_{1},...,x_{c}]/(f_{1},...,f_{m})$.

Since $X_{\mathbb{Q}}$ is an $(LCI)$, we have an affine open cover
$\{\beta_{i}:U_{i}\hookrightarrow X_{\mathbb{Q}}\}_{i}$ of $X_{\mathbb{Q}}$
with inclusions $\psi_{i}:U_{i}\hookrightarrow\mathbb{A}_{\mathbb{Q}}^{M_{i}}$
and maps $\varphi_{i}:\mathbb{A}_{\mathbb{Q}}^{M_{i}}\longrightarrow\mathbb{A}_{\mathbb{Q}}^{N_{i}}$,
flat over $0$, such that $\psi_{i}:U_{i}\simeq\varphi_{i}^{-1}(0)$.
We may assume that $U_{i}$ is isomorphic to a basic open set $D(g_{i})$
for $g_{i}\in\mathbb{Q}[X]$ and $\beta_{i}^{*}:\mathbb{Q}[X]\longrightarrow\mathbb{Q}[X,t]/(g_{i}t-1)$
is the natural map. Since $\{D(g_{i})\}_{i}$ is a cover of $X_{\mathbb{Q}}$,
there exist $c_{i}'\in\mathbb{Z}[X]$ and $d_{i}\in\mathbb{Z}$ such
that $\sum\frac{c_{i}'\cdot g_{i}}{d_{i}}=1$. Thus, by multiplying
by all the $d_{i}$'s, we obtain $\sum c_{i}g_{i}=D$ for some $c_{i}\in\mathbb{Z}[X]$
and $D\in\mathbb{Z}$.

Choose large enough $P\in\mathbb{N}$ such that the following algebra
\[
\mathbb{Z}[x_{1},...,x_{c},t]/(f_{1},...,f_{m},Pg_{i}t-D\cdot P))
\]
 is a coordinate ring of a $\mathbb{Z}$-scheme $\widetilde{U}_{i}$,
for any $i$. Moreover, notice that $\widetilde{U}_{i}\simeq U_{i}$
over $\mathbb{Q}$. 
\begin{lem}
There exists $N\in\mathbb{N}$, such that for any prime power $q=p^{r}$
and any $n>N$ we have $\left|X(\mathbb{Z}_{q}/\mathfrak{m}_{q}^{n})\right|\leq\sum_{i}\left|\widetilde{U}_{i}(\mathbb{Z}_{q}/\mathfrak{m}_{q}^{n})\right|$. 
\end{lem}

\begin{proof}
Let $N(p)$ be the maximal integer such that $p^{N(p)}|D\cdot P$.
We first claim that for any $n>N(p)+1$ and $(a_{1},...,a_{c})\in X(\mathbb{Z}_{q}/\mathfrak{m}_{q}^{n})$,
there exists some $i$ such that $Pg_{i}(a_{1},...,a_{c})\notin\mathfrak{m}_{q}^{N(p)+1}/\mathfrak{m}_{q}^{n}$.
Indeed, if $Pg_{i}(a_{1},...,a_{c})\in\mathfrak{m}_{q}^{N(p)+1}/\mathfrak{m}_{q}^{n}$
for any $i$, then $\sum Pg_{i}(a_{1},...,a_{c})\cdot c_{i}(a_{1},...,a_{c})=D\cdot P\in\mathfrak{m}_{q}^{N(p)+1}/\mathfrak{m}_{q}^{n}$
and hence $p^{N(p)+1}|D\cdot P$ leading to a contradiction. Set $N:=D\cdot P+1$
and notice that $N>N(p)+1$ for any prime $p$. 

Fix $n>N$ and let $i$ such that $Pg_{i}(a_{1},...,a_{c})\notin\mathfrak{m}_{q}^{N(p)+1}/\mathfrak{m}_{q}^{n}$.
We now claim that the equation $Pg_{i}(a_{1},...,a_{c})t-PD=0$ has
a solution in $\mathbb{Z}_{q}/\mathfrak{m}_{q}^{n}$. Indeed, if $Pg_{i}(a_{1},...,a_{c})$
is invertible in $\mathbb{Z}_{q}/\mathfrak{m}_{q}^{n}$, we are done.
Otherwise, we have that $Pg_{i}(a_{1},...,a_{c})=p^{l}\cdot b\in\mathfrak{m}_{q}^{l}/\mathfrak{m}_{q}^{n}$
for some $l\leq N(p)$, where $b$ is invertible. Write $PD=p^{l}\cdot a$.
We can rewrite the equation as $p^{l}\cdot(bt-a)=0$, which has a
solution $d\in\mathbb{Z}_{q}/\mathfrak{m}_{q}^{n}$  since $b$ is
invertible. We see that for any $n>N$ and any $(a_{1},...,a_{c})\in X(\mathbb{Z}_{q}/\mathfrak{m}_{q}^{n})$
there exists $i$ and $d\in\mathbb{Z}_{q}/\mathfrak{m}_{q}^{n}$ such
that $(a_{1},...,a_{c},d)\in\widetilde{U}_{i}(\mathbb{Z}_{q}/\mathfrak{m}_{q}^{n})$.
This implies the lemma.
\end{proof}
Since $\left(\widetilde{U}_{i}\right)_{\mathbb{Q}}\simeq U_{i}$ is
a $(CIA)$ for any $i$, we obtain: 
\[
\frac{\left|X(\mathbb{Z}_{q}/\mathfrak{m}_{q}^{n})\right|}{q^{n\mathrm{dim}X_{\mathbb{Q}}}}\leq\sum_{i}\frac{\left|\widetilde{U}_{i}(\mathbb{Z}_{q}/\mathfrak{m}_{q}^{n})\right|}{q^{n\mathrm{dim}X_{\mathbb{Q}}}}<\sum C_{i},
\]
where $C_{i}=\underset{n}{\mathrm{sup}}h_{\widetilde{U}_{i}}(\mathbb{Z}_{q}/\mathfrak{m}_{q}^{n})$.
The implication $iii)\Longrightarrow iv')$ of Theorem \ref{Main extended theorem}
now follows.

\end{document}